\documentclass[11pt,reqno]{amsart}     

\usepackage{euscript,amssymb,amsbsy,mathabx}
\usepackage[arrow,matrix,curve]{xy}
\usepackage{a4wide,graphicx,longtable}

\newtheorem{stuff}{Stuff}[section]
\newtheorem{theorem}[stuff]{\bf Theorem}

\newtheorem{lemma}[stuff]{\bf Lemma}
\newtheorem{corollary}[stuff]{\bf Corollary}

\newenvironment{definition}{%
\vskip1ex\refstepcounter{stuff}\trivlist \itemindent 0pt
\item[\hskip\labelsep\bf Definition \thestuff.]%
\ignorespaces}{\endtrivlist\vskip1ex}%
\font\tenmsa=msam10 %
\newcommand\hdashpiece{%
{\vrule height2.75pt depth-2.35pt width2.3pt \kern1.7pt}}%
\newcommand\hdashpieces{%
{\hdashpiece\hdashpiece\hdashpiece\hdashpiece}}%
\newcommand\dashto{\mathrel{%
\hdashpiece\hdashpiece\kern-0.4pt\hbox{\tenmsa K}}}%
\newcommand\dashar{\mathrel{%
\hdashpieces\kern-0.4pt\hbox{\tenmsa K}}}%
 
\let\euf\EuScript 
\let\cal\mathcal 
\let\mbb\mathbb

\let\bsymb\boldsymbol 
\DeclareFontFamily{OT1}{rsfs}{} 
\DeclareFontShape{OT1}{rsfs}{n}{it}{<->rsfs10}{} 
\DeclareMathAlphabet{\crl}{OT1}{rsfs}{n}{it} 
 
\let\tld\tilde 
\let\wtld\widetilde 
 
\let\nit\noindent

\let\srel\stackrel 
\let\vphi\varphi 
 
\let\veps\varepsilon

\newcommand\ra{{a}}
\newcommand\rd{{\rm d}} 
 
\newcommand\ev{{\rm ev}}

\newcommand\Pic{\mathop{\rm Pic}\nolimits} 
 
\newcommand\Ker{{\rm Ker}} 
 
\newcommand\invq{{\slash\kern-0.65ex\slash}}

\let\d\delta

\let\si\sigma

\let\les\leqslant 
\let\ges\geqslant

\newcommand\bone{{1\kern-0.57ex\rm l}} 
 
\newcommand\cA{{\crl A}}

\newcommand\Sym{\mathop{\rm Sym}\nolimits}

\newcommand\mx{{\rm max}}

\newcommand{\tS}{{\tilde S}} 
\newcommand{\res}{{\rm res}} 
\newcommand{\ouset}[3]{\underset{#1}{\overset{#2}{#3}}}

\newcommand\mlt{{\rm mult}}
\newcommand{\PP}{{P}}

\author{Mihai Halic}
\keywords{$K3$ surfaces; nodal curves; Wahl map; Seshadri constants}
\subjclass[2010]{14H10; 14J28; 14D15; 14C20} 
\address{
\begin{minipage}{30em}
\end{minipage}
}

\begin{document} 
 
 \title{About the surjectivity of the Wahl map\\ of nodal curves on $K3$ surfaces\\ 
Corrigendum to:\\ Modular properties of nodal curves on \mbox{$K3$} surfaces} 

\begin{abstract}
I correct an error in \cite{ha} concerning the non-surjectivity of the Wahl map of 
nodal curves on $K3$ surfaces. 
I also obtain a lower bound of independent interest for the multiple point Seshadri 
constants of general $K3$ surfaces. 
\end{abstract}

\maketitle
 
\markboth{\sc Mihai Halic}%
{\sc Corrigendum: Modular properties of nodal curves on $K3$ surfaces} 

\section{The correction}

The goal of this note is twofold: 
\begin{enumerate}
	\item First, the proof of \cite[Theorem 3.1]{ha} is incorrect: 
	the fault is at the Step 2 of the proof. 
	In the meantime, the result has been proved in \cite{kem} with better bounds. 
	\item Second, I correct \cite[Theorem 4.2]{ha}. At pp. 884, the last row of the 
	diagram A.1 should be tensored by $\cal O_E(-2E)$. This error affects the 
	subsequent computations from Lemma A.2 onward, which are used in the proof 
	of the Theorem. 
\end{enumerate}
Recall that $\hat C\in|\cal L=\cA^d|$ is a nodal curve with nodes 
$\euf N:=\{\hat x_1,\ldots,\hat x_\delta\}$ on the polarized $K3$ surface $(S,\cA)$, 
such that $\cA\in\Pic(S)$ is not divisible, and $\cA^2=2(n-1)$. 
(The paper \cite{ha} deals only with $K3$ surfaces with cyclic Picard group). 
Let $\si:\tld S\to S$ be the blow-up of $S$ at $\euf N$, and denote by 
$E^\ra$, $\ra =1,\ldots,\delta$, the exceptional divisors, and $E:=E^1+\ldots+E^\delta$. 
The normalization $C$ of $\hat C$ fits into 
\begin{equation}\label{eq:C}
\scalebox{.8}{$
\xymatrix@R=1.5em@C=5em{ 
(C,\Delta)\;\ar[d]_-\nu\ar@{^{(}->}[r]^{\tilde u}\ar[dr]^-u& 
\tilde S\ar[d]^-\si 
\\ 
(\hat C,\euf N)\;\ar@{^{(}->}[r]^{j}& 
S, 
} 
$}
\end{equation}
$\tilde u$ is an embedding, and $K_C=\si^*\cal L(-E)\otimes\cal O_C$. 
The curve $C$ carries the divisor 
\\ \centerline{
$\Delta:=x_{1,1}+x_{1,2}+\ldots+x_{\d,1}+x_{\d,2}$,
} 
where $\{x_{\ra ,1},x_{\ra ,2}\}=E^\ra \cap C$ is the pre-image of 
$\hat x_a\in\hat C$ by $\nu$. 

In general, if $V'$ is a subscheme of some variety $V$, $\cal I_V(V')$ or $\cal I(V')$ 
stands for its sheaf of ideals, and $DV'\subset V\times V$ denotes the diagonally 
embedded $V'$. 

Let $(X,\Delta_X)$ be an arbitrary smooth, irreducible curve together with $\delta$ 
pairwise disjoint pairs of points 
$
\Delta_X=\{\{x_{1,1},x_{1,2}\},\ldots,\{x_{\delta,1},x_{\delta,2}\}\}\subset X.
$ 
The exact sequence $0\to\cal I(DX)^2\to\cal I(DX)\to K_X\to 0$ yields 
the Wahl map  \\[.5ex]
\centerline{
$w_X:H^0(X\times X,\cal I(DX)\otimes K_{X\times X}){\to} H^0(X,K_X^3).$ 
}\\[.5ex]
The vector space $H^0(\cal I(DX)\otimes K_{X\times X})$ splits into \\[-.5ex]
\centerline{
$H^0(\cal I(DX)\otimes K_{X\times X})\cap\Sym^2H^0(K_X) 
\oplus\overset{2}\bigwedge H^0(K_X),$
}\\[.5ex] 
and $w_X$ vanishes on the first direct summand, as it is skew-symmetric. 
Denote 
\begin{equation}\label{eq:Pi}
\PP_{\Delta_X}:=\mbox{$\ouset{a=1}{\delta}{\bigcup}$}
\{x_{a,1},x_{a,2}\}\times\{x_{a,1},x_{a,2}\}\subset X\times X,
\end{equation} 
and let $w_{X,\Delta_X}$ be the restriction of $w_X$ to 
$H^0\big(\cal I(\PP_{\Delta_X})\cdot\cal I(DX)\otimes K_{X\times X}\big)
\cap\overset{2}{\bigwedge}H^0(K_X)$. 
(Thus $w_{X,\Delta_X}$ is a punctual modification of the usual Wahl map.) 
With these notations, we replace \cite[Theorem 4.2]{ha} by the following. 

\begin{theorem}\label{thm:main} 
\nit{\rm (i)} Let $(S,\cA)$, $\cA^2\ges6$, be as above. 
Consider a nodal curve $\hat C\,{\in}\,|d\cA|$ with 
\begin{equation}\label{eq:dta}
\mbox{$\delta\les\min\bigl\{\frac{d^2\cA^2}{3(d+4)},\delta_\mx(n,d)\bigr\}$}
\end{equation} 
nodes and let $(C,\Delta)$ be as above ($\delta_\mx(n,d)$ is defined 
in \cite[pp. 872]{ha};  the minimum is the first expression, except a finite 
number of cases). Then the homomorphism $w_{C,\Delta}$ is not surjective.  

\nit{\rm (ii)} For generic a generic curve $X$ of genus $g\ges 12$ with generic 
markings $\Delta_X$, such that $\delta\les\frac{g-1}{2}$, the homomorphism 
$w_{X,\Delta_X}$ is surjective.
\end{theorem}

This is a non-surjectivity property for {\em the pair} $(C,\Delta)$, rather than 
for $C$ itself. I conclude the note (see Section \ref{sct:final}) with some evidence 
towards  the non-surjectivity of the Wahl map $w_C$ itself, and comment on 
related work in \cite{kem}.


\section{Relationship between the Wahl maps of \mbox{$C$} and \mbox{$\tS$}}
\label{sct:CS}

\begin{lemma}\label{lm:LD}
{\rm (i)} The following diagram has exact rows and columns: 
\begin{equation}\label{eq:DD}
\scalebox{.8}{$
\xymatrix@R=1.5em@C=2em{
0\ar[r]
&
\cal I_{C\times C}(DC)^2
\ar[r]\ar@{=}[d]
&
\cal I_{C\times C}(D\Delta)\cdot\cal I_{C\times C}(DC)
\ar[r]^-{w'_{C}}\ar@{^(->}[d]
&
K_C(-\Delta)
\ar[r]\ar@{^(->}[d]
&0 
\\
0\ar[r]
&
\cal I_{C\times C}(DC)^2
\ar[r]
&
\cal I_{C\times C}(DC)
\ar[r]^-{w_C}\ar@{->>}[d]^-{\ev_\Delta}
&
K_C
\ar[r]\ar@{->>}[d]^-{\ev_\Delta}
&0 
\\ 
&&
\cal I_{C\times C}(DC)_{D\Delta}
\ar@{=}[r]
&
K_{C,\Delta}
&
}
$}
\end{equation}
In other words, we have $\cal I:=
\cal I_{C\times C}(D\Delta)\cdot\cal I_{C\times C}(DC)=w_C^{-1}(K_C(-\Delta))$. 
Also, the involution $\tau_C$ which interchanges the factors of $C\times C$ leaves 
$\cal I$ invariant. 

\nit{\rm (ii)} $H^0\bigl(
C\times C,
\cal I\otimes K_{C\times C})
\bigr)
=
w_C^{-1}\bigl(\;H^0(C,K_C^3(-\Delta))\;\bigr)$. 

\nit{\rm (iii)} For 
$
\Lambda:=
\bigl\{
s-\tau_C^*(s)\mid 
s\in 
H^0\bigl(
\cal I\otimes K_{C\times C})
\bigr)
\bigr\}
$ 
holds \\ 
\centerline{$
\Lambda
\srel{(\star)}{=}
H^0\bigl(\cal I\otimes K_{C\times C}\bigr)
\;\cap\;
\mbox{$\overset{2}{\bigwedge}$}\,H^0(K_C)
\srel{(\star\star)}{\subset}
H^0\bigl(\cal I_{C\times C}(DC)\otimes K_{C\times C}\bigr).
$}

\nit{\rm (iv)} 
$
w'_{C}(\Lambda)
=
w'_{C}\bigl(\;
H^0\bigl(\cal I\otimes K_{C\times C}\bigr)\;\bigr).
$ 
\end{lemma}

\begin{proof}
(i) The middle column is exact because $\cal I_{C\times C}(DC)$ is locally free. 
We check the exactness of the first row around each point $(o,o)\in D\Delta$. 
Let $u$ be a local (analytic) coordinate on $C$ such that $o=0$, and $u_1,u_2$ 
be the corresponding coordinates on $C\times C$. Then the first row becomes 
$
0\to\langle u_2-u_1\rangle^2\to
\langle u_2-u_1\rangle\cdot\langle u_1,u_2\rangle\to
\langle u\rangle\cdot\rd u\to 0,
$ 
with $\rd u:=(u_2-u_1){\rm mod} (u_2-u_1)^2$, which is exact. 
The second statement is obvious. 

\nit(ii) We tensor \eqref{eq:DD} by $K_{C\times C}$, and take the sections 
in the last two columns. An elementary diagram chasing yields the claim. 

\nit(iii) Let us prove $(\star)$. 
The vector space $\Lambda$ is contained in 
$\mbox{$\overset{2}{\bigwedge}$}\,H^0(K_C)$ by the very definition, and also in 
$H^0\bigl(\cal I\otimes K_{C\times C}\bigr)$ because  the sheaf $\cal I(C,\Delta)$ 
is $\tau_C$-invariant. For the inclusion in the opposite direction, take $s$ in the 
intersection. As $s\in\mbox{$\overset{2}{\bigwedge}$}\;H^0(K_C)$, it follows 
$\tau_C^*(s)=-s$, so $s=1/2\cdot(s-\tau_C^*(s))\in\Lambda$. 
The inclusion $(\star\star)$ is obvious. 

\nit(iv) Indeed, the Wahl map is anti-commutative: 
$w_C(\underset{i}{\sum}s_i\otimes t_i)=-w_C(\underset{i}{\sum}t_i\otimes s_i)$. 
\end{proof}

\begin{lemma}\label{lm:LP} 
Let 
$\Xi:=\{(x_{\ra ,1},x_{\ra ,2}),(x_{\ra ,2},x_{\ra ,1})\mid
\ra=1,\ldots,\delta\}\subset C\times C$, 
and consider the sheaf of ideals 
$\cal I(C,\Delta)\!:=\cal I(\Xi)\!\cdot\!\cal I
\srel{\eqref{eq:Pi}}{=}
\cal I(\PP_\Delta)\!\cdot\!\cal I(DC)\subset\!\cal I\!.$ 
Furthermore, denote 
\begin{equation}\label{eq:LD}
\Lambda(\Delta)\!:=H^0(\cal I(C,\Delta)\otimes K_{C\times C})
\,\cap\,
\mbox{$\overset{2}{\bigwedge}$}\,H^0(K_C).
\end{equation}  
Then the following statements hold:

\nit{\rm (i)} $\cal I(C,\Delta)$ is $\tau_C$-invariant, so 
$w'_C\bigl(\,\Lambda(\Delta)\,\bigr)
=w'_C\bigl(\;H^0(\cal I(C,\Delta)\otimes K_{C\times C})\;\bigr).$
\vskip1ex

\nit{\rm (ii)}  
$\cal I(C,\Delta)+\cal I_{C\times C}(DC)^2=\cal I$, 
and $\cal I(C,\Delta)\cap\cal I(DC)^2=\cal I(\Xi)\cdot\cal I(DC)^2$. Therefore 
the various sheaves introduced so far fit into the commutative diagram 
\begin{equation}\label{eq:CD}
\scalebox{.8}{$
\begin{array}{l}
\xymatrix@R=1.5em@C=2em{
0\ar[r]&
\cal I(\Xi)\cdot\cal I(DC)^2\ar[r]\ar@{^(->}[d]& 
\cal I(C,\Delta)\ar[r]^-{w_{C,\Delta}'}\ar@{^(->}[d]&
K_C(-\Delta)
\ar[r]\ar@{=}[d]&
0\\
0\ar[r]&
\cal I(DC)^2\ar@{->>}[d]^-{\ev_\Xi''}\ar[r]& 
\cal I\ar@{->>}[d]^-{\ev'_\Xi}\ar[r]^-{w'_{C}}&
K_C(-\Delta)\ar[r]&
0\\
&\cal O_{\Xi}\ar[r]^-\cong&\cal O_{\Xi}&&
}
\\[.75ex] 
\text{
(The homomorphism $w_{C,\Delta}$ in the introduction equals 
$H^0(w_{C,\Delta}')$, defined after tensoring by $K_{C\times C}$.)
}
\end{array}
$}
\end{equation}
\end{lemma}

\begin{proof}
(i) The proof is identical to Lemma \ref{lm:LD}(iv). 

\nit(ii) The inclusion $\subset$ is clear. For the reverse, notice that 
$\cal O_{C\times C}=\cal I(\Xi)+\cal I$, so 
$\cal I\kern-.35ex\subset\kern-.35ex\cal I(C,\Delta)+\cal I^2
\kern-.35ex\subset\kern-.35ex\cal I(C,\Delta)+\cal I(DC)^2$. 
The second claim is analogous.
\end{proof}

Now we compare the Wahl maps of $C$ and $\tS$. 
Let $\rho\!:\!\cal I_{\tld S\times\tld S}(D\tld S)\to \cal I_{C\times C}(DC)$ 
be the restriction homomorphism, and $\cal M\!:=\si^*\cal L(-E)$. The diagram below 
relates various objects involved in the definition of $w_C$ and $w_{\tld S}$: 
\begin{equation}\label{eq:W}
\hspace{-4ex}
\scalebox{.7}{$
\xymatrix@R=1.25em@C=1em{
&0\ar[d]&0\ar[d]&0\ar[d]&
\\ 
0\ar[r]&
\cal O_{\tld S}(-E)\ar[r]\ar[d]&
\cal O_{\tld S}(E)\ar[r]\ar[d]&
\cal O_{2E}(E)\ar[r]\ar[d]&
0
\\ 
0\ar[r]&
\cal M(-2E)\ar[r]\ar[d]&
\cal M\ar[r]\ar[d]&
\cal M\otimes\cal O_{2E}\ar[r]\ar[d]&
0
\\ 
0\ar[r]&
K_C(-2\Delta)\ar[r]\ar[d]&
K_C\ar[r]\ar[d]&
K_C\otimes\cal O_{2\Delta}\ar[r]\ar[d]&
0
\\ 
&0&0&0&
}
$}
\hspace{1ex}
\scalebox{.7}{$
\xymatrix@R=1.75em@C=1em{
0\ar[r]\ar[d]&
H^0(\cal O_{\tld S}(E))=\mbb C\tld s_E\ar[r]\ar@{^(->}[d]&
H^0(\cal O_{2E}(E))\ar[r]\ar@{^(->}[d]&
H^1(\cal O_{\tld S}(-E))\ar[r]&
0
\\ 
H^0(\cal M(-2E))\ar@{^(->}[r]\ar[d]&
H^0(\cal M)\ar[r]\ar@{->>}[d]&
H^0(\cal M\otimes\cal O_{2E})\ar[r]\ar@{->>}[d]&
H^1(\cal M(-2E))\ar[r]&
\ldots
\\ 
H^0(K_C(-2\Delta))\ar@{^(->}[r]\ar[d]&
H^0(K_C)\ar[r]&
\mbox{$\underset{x\in\Delta}{\bigoplus}$}K_{C,2x}\ar[r]&&
\\ 
H^1(\cal O_{\tld S}(-E))&&&&
}
$}
\end{equation}
The rightmost column corresponds to the first order expansions of the sections along 
$E$ and at $\Delta$. By using $0{\to}\cal O_E(1){\to}\cal O_{2E}{\to}\cal O_E{\to}0$, 
we deduce that it fits into:
\begin{equation}\label{eq:2E}
\scalebox{.8}{$
\xymatrix@R=1.5em@C=2em{
H^0(\cal O_E)\cong\mbb C^\delta\ar@{^(->}[r]\ar[d]^-\cong
&
H^0(\cal O_E(2))\cong\mbb C^{3\delta}\ar@{->>}[r]\ar@{^(->}[d]
&
H^0(\cal O_\Delta)\cong\mbb C^{2\delta}\ar@{^(->}[d]
\\ 
H^0(\cal O_{2E}(E))\ar@{^(->}[r]&
H^0(\cal M\otimes\cal O_{2E})\ar@{->>}[r]\ar@{->>}[d]&
H^0(K_C\otimes\cal O_{2\Delta})\ar@{->>}[d]
\\ 
&
H^0(\cal O_E(1))\cong\mbb C^{2\delta}\ar[r]^-\cong
&
H^0(\cal O_\Delta)\cong\mbb C^{2\delta}
}
$}
\end{equation}
Along each $E^\ra\subset\tld S$, we consider local coordinates $u,v$ as follows: 
$v$ is the coordinate along $E^\ra$, and $u$ is a coordinate in the normal direction 
to $E^\ra$ (so $E^\ra$ is given by $\{u=0\}$). 
Moreover, we assume that $C$ is given by $\{v=0\}$ around the intersection points 
$\{x_{\ra ,1},x_{\ra ,2}\}=E^\ra \cap C$. 
Then any element $\tld s\in H^0(\cal M)$ can be expanded as 
\begin{equation}\label{eq:tlds}
\tld s=\tld s^\ra _0(v)+u\tld s^\ra _1(v)+O(u^2),
\end{equation} 
and its image in $H^0(\cal M\otimes\cal O_{2E^\ra })$ is 
$\tld s^\ra _0(v)+u\tld s^\ra _1(v)$. Finally, observe that the values 
of $w_{\tld S}$ are sections of $\Omega^1_{\tld S}\otimes\cal M^2$, 
and the restriction of this latter to $E$ fits into \\[.5ex]
\centerline{
$
0\to
\underbrace{\cal O_E(3)}%
_{\text{\scriptsize normal component}}
\to
\Omega^1_{\tld S}\otimes\cal M^2|_{E}\to
\underbrace{\Omega^1_E\otimes\cal M^2_E=\cal O_E}%
_{\text{\scriptsize tangential component}}
\to 0. 
$
}

\begin{lemma}\label{lm:E}
Let the notations be as in \eqref{eq:tlds}. We consider 
$\tld e=
\underset{i}{\sum}\,\tld s_i\wedge\tld t_i\in\overset{2}{\bigwedge}\,H^0(\cal M)$, 
and let $e:=\rho(\tld e)=\underset{i}{\sum}\, s_i\wedge t_i$. 
Then the following statements hold:

\nit{\rm(i)} 
$w_\tS(\tld e)\in H^0(\Omega^1_{\tld S}(-E)\otimes\cal M^2)$ if and only if: 
\\[.5ex]\centerline{
$
\left\{
\begin{array}{ccl}
(\star)
&\quad&
\underset{i}{\sum}\,
\bigl(
s_{i}(x_{\ra ,1}) t_{i}(x_{\ra ,2})
-
t_{i}(x_{\ra ,1}) s_{i}(x_{\ra ,2})
\bigr)
=0,\quad\text{and}
\\[1ex] 
(\star\star)
&\quad&
\underset{i}{\sum}\,
\bigl(
\tld s_{i,0}^\ra \tld t_{i,1}^\ra -\tld t_{i,0}^\ra \tld s_{i,1}^\ra 
\bigr)
=0,\quad\forall\,a=1,\ldots,\delta.
\end{array}
\right.
$
}

\nit{\rm (ii)}  
$\Lambda(E):=
w_{\tld S}^{-1}\bigl(H^0(\Omega^1_{\tld S}(-E)\otimes\cal M^2)\bigr)
\,\cap\,\overset{2}{\bigwedge}H^0(\cal M)$ has the property 
\\ \centerline{
$
w_{\tld S}(\Lambda(E))
=
w_{\tld S}\bigl(\;
w_{\tld S}^{-1}\bigl(H^0(\Omega^1_{\tld S}(-E)\otimes\cal M^2)\bigr)
\;\bigr).
$
}
\nit{\rm (iii)} 
$\rho\big(\,\Lambda(E)\,\big)\subset\Lambda(\Delta)$, where the right hand side is 
defined by \eqref{eq:LD}. 
\end{lemma}

\begin{proof}
(i) The element $w_{\tld S}(\tld e)$ vanishes along $E$ if and only if both its 
tangential and normal components along each $E^\ra \subset E$ vanish. 
A short computation shows that the normal component is $(\star\star)$. 
The tangential component is 
$\underset{i}{\sum}(
\tld s^\ra _{i,0}(\tld t^\ra _{i,0})'-\tld t^\ra _{i,0}(\tld s^\ra _{i,0})'
)$. But $\tld s^\ra _{i,0},\tld t^\ra _{i,0}\in H^0(\cal O_{E^\ra }(1))$, 
that is they are linear polynomials in $v$, so\\[.5ex] 
\centerline{$
\tld s^\ra _{i,0}(\tld t^\ra _{i,0})'-\tld t^\ra _{i,0}(\tld s^\ra _{i,0})' 
=
\tld s^\ra _{i,0}(x_{\ra ,1})\tld t^\ra _{i,0}(x_{\ra ,2})
-\tld t^\ra _{i,0}(x_{\ra ,1})\tld s^\ra _{i,0}(x_{\ra ,2}),
$}\\[.5ex] 
up to a constant factor. Also, we have 
$\tld s^\ra _{i,0}(x_{\ra ,j})=\tld s^\ra (x_{\ra ,j})
=s(x_{\ra ,j})$, and $(\star)$ follows. 

\nit(ii) The vector space 
$w_{\tld S}^{-1}\bigl(H^0(\Omega^1_{\tld S}(-E)\otimes\cal M^2)\bigr)$ 
is invariant under the involution $\tau_{\tld S}$ of $\tld S\times\tld S$ 
which switches the two factors.  As $w_{\tld S}$ is anti-commutative, 
the claim follows as in Lemma \ref{lm:LD}. 

\nit(iii) Take $\tld e\in\Lambda(E)$ and $e:=\rho(\tld e)$. 
Then $e(x_{\ra ,1},x_{\ra ,2})=-e(x_{\ra ,2},x_{\ra ,1})\srel{(\star)}{=}0$, and 
also $w_C(e)(x_{\ra ,j})$ equals the expression $(\star\star)$ at $x_{\ra ,j}$ 
(so it vanishes), for $j=1,2$. 
\end{proof}

Now we consider the commutative diagram:
\begin{equation}{\label{nodal-fc}} 
\scalebox{.8}{$
\xymatrix@R=2em@C=3em{ 
\Lambda(E)\ar[r]^-{w_{\tld S}}\ar[d]_{\rho_\Delta}
^-{\mbox{\kern-1ex\small%
\begin{tabular}{l}\it cf.\\ \ref{lm:E}\text{(iii)}\end{tabular}}}
& 
H^0\bigl(\tS,\Omega^1_\tS(-E)\otimes\cal M^2\bigr) 
\ar[r]^-{\res_C}\ar[d]& 
H^0\bigl(C,\Omega^1_\tS|_C\otimes K_C^2(-\Delta)\bigr) 
\ar[dl]^-b 
\\  
\Lambda(\Delta)\ar[r]^(0.42){w_{C,\Delta}}& 
H^0\bigl(C,K_C^3(-\Delta)\bigr) 
&} 
$}
\end{equation} 
It is the substitute in the case of nodal curves for \cite[diagram (4.2)]{ha}. 

\begin{lemma}\label{lm:surj} Assume $\Pic(S)=\mbb Z\cA$. 
Then $\rho_\Delta:\Lambda(E)\to\Lambda(\Delta)$ is surjective.
\end{lemma}

\begin{proof}
The restriction $H^0(\tS,\cal M)\to H^0(C,K_C)$ is surjective 
(see \cite[Lemma A.1]{ha}), and the kernel of 
$
\overset{2}{\bigwedge}\,H^0(\tld S,\cal M)\to
\overset{2}{\bigwedge}\,H^0(C,K_C)
$ 
consists of elements of the form $\tld t\wedge (\tld s_C\tld s_E)$, where 
$\tld t\in H^0(\cal M)$ and $\tld s_C,\tld s_E$ are the canonical sections of 
$\cal O_\tS(C)$ and $\cal O_\tS(E)$ respectively. 
(See the middle column of \eqref{eq:W}.) 

Consider $e=\underset{i}{\sum}\;(s_i\otimes t_i-t_i\otimes s_i)
\in\Lambda(\Delta)$, and let 
$\tld e=\underset{i}{\sum}\;(\tld s_i\otimes\tld t_i-\tld t_i\otimes\tld s_i)
\in\overset{2}{\bigwedge}\,H^0(\cal M)$ 
be such that $\rho(\tld e)=e$. The proof of \ref{lm:E}(i) shows that, for all $\ra $, the 
tangential component of $w_\tS(\tld e)|_{E^\ra }$ equals $e(x_{\ra ,1},x_{\ra ,2})=0$, 
so $w_\tS(\tld e)|_{E}$ is a section of 
$\Omega^1_{E/\tS}\otimes\cal M_E^2\cong\cal O_E(3)$. 
Since $w_\tS(\tld e)|_{E}$ vanishes at the points of $\Delta$, it is actually determined up 
to an element in $H^0(\cal O_E(1))$. We claim that this latter can be cancelled by adding 
to $\tld e$ a suitable element of the form $\tld t\wedge (\tld s_C\tld s_E)$. 
A short computation yields 
\\[.5ex] \centerline{
$
w_\tS(\,\tld t\wedge (\tld s_C\tld s_E)\,)|_E
=
\tld t_E\cdot(\tld s_C|_E)\cdot (\rd\tld s_E)|_E\in \cal O_E(3),
$
}\\[.5ex]
where $\tld t_E\in H^0(\cal O_E(1))$, $\tld s_C|_E\in H^0(\cal O_E(2))$ 
vanishes at $\Delta=E\cap C$, and $(\rd\tld s_E)|_E\in H^0(\cal O_E)$ 
(it is a section of $\Omega^1_\tS|_E$ with vanishing tangential component). 
Thus these two latter factors are actually (non-zero) scalars. 

The previous discussion shows that 
$\tld e+\tld t\wedge (\tld s_C\tld s_E)\in\Lambda(E)$ 
as soon as $\tld t\in H^0(\cal M)$ satisfies 
$\tld t_E=-w_\tS(\tld e)|_E\in H^0(\cal M_E)$. According to 
Corollary \ref{cor:0}, such an element $\tld t$ exists because 
the restriction $H^0(\cal M)\to H^0(\cal M_E)$ is surjective. 
\end{proof}

\begin{proof}(of Theorem \ref{thm:main}) (i) 
{\it Case $\Pic(S)=\mbb Z\cA$.}\quad 
If $w_{C,\Delta}$ is surjective, then the homomorphism $b$ in the diagram 
\eqref{nodal-fc} is surjective too. 
Now we follow the same pattern as in \cite[pp. 884, top]{ha}: 
$b$ is the restriction homomorphism at the level of sections of \\ 
\centerline{
$
0\to K_C\to\Omega^1_\tS\bigr|_C\otimes K_C^2(-\Delta) 
\to K_C^3(-\Delta)\to 0,
$
}
and its surjectivity implies that this sequence splits. 
This contradicts \cite[Lemma 4.1]{ha}. 

\nit{\it General case.}\label{gen-case}\quad 
It is a deformation argument. We consider \\ 
\centerline{$
\begin{array}{l}
\cal K_n:=
\{(S,\cA)\mid\cA\in\Pic(S)\text{ is ample, not divisible},\cA^2=2(n-1)\},
\\ 
\cal V^d_{n,\delta}:=\{((S,\cA),\hat C)\mid(S,\cA)\in\cal K_n,\;
\hat C\in|d\cA|\text{ nodal curve with }\delta\text{ nodes}
\}.
\end{array}
$}
Then the natural projection $\kappa:\cal V^d_{n,\delta}\to\cal K_n$ is 
submersive onto an open subset of $\cal K_n$. (See \cite[Theorem 1.1(iii)]{ha} 
and the reference therein.) 

Hence for any $((S,\cA),\hat C)\in\cal V^d_{n,\delta}$ there is a smooth 
deformation $((S_t,\cA_t),\hat C_t)$ parameterized by an open subset 
$T\subset\cal K_n$. The points $t\in T$ such that $\Pic(S_t)=\mbb Z\cA_t$ are dense; 
for these $w_{C_t,\Delta_t}$ are non-surjective. Since the non-surjectivity condition 
is closed, we deduce that $w_{C,\Delta}$ is non-surjective too.  
\medskip

\nit(ii) Now let $(X,\Delta_X)$ be a generic marked curve of genus at least $12$. 
By \cite{chm}, the Wahl map 
$w_X\,{:}\,\overset{2}{\bigwedge}\,H^0(K_X)\to H^0(K_X^3)$ is surjective, 
thus $\wtld w\,'_X\,{:=}\,H^0(w'_X\otimes K_{X\times X})$ 
in \eqref{eq:CD} is surjective as well (see lemma \ref{lm:LD}(ii)). 
As $\delta\les\frac{g-1}{2}$, the evaluation homomorphism 
$H^0(K_X)\to K_{X}\otimes\cal O_{\Delta_X}$ is surjective for 
generic markings, so the same holds for 
\\ \centerline{$
H^0(K_X)^{\otimes 2}\to
\ouset{a=1}{\delta}{\bigoplus}(K_{X,x_{a,1}}\oplus K_{X,x_{a,2}})^{\otimes 2}.
$}
The restriction to the anti-symmetric part (on both sides) yields the surjectivity of 
\\ \centerline{
$\ev_\Xi:\overset{2}{\bigwedge}H^0(K_X)\to
\ouset{a=1}{\delta}{\bigoplus}K_{X,x_{a,1}}\otimes K_{X,x_{a,2}}
=\ouset{a=1}{\delta}{\bigoplus} K_{X\times X,(x_{a,1},\,x_{a,2})}.
$
}
(For $s\in\overset{2}{\bigwedge}H^0(K_X)$,  $\ev_\Xi(s)$ takes opposite values 
at $(x_{a,1},x_{a,2})$ and $(x_{a,2},x_{a,1})$.) 

The diagram \eqref{eq:CD} yields 
\begin{equation}\label{eq:X}
\hspace{-3ex}\scalebox{.8}{$
\xymatrix@R=1.5em@C=1.55em{
H^0\big(\cal I(\Xi)\cdot\cal I(DX)^2\otimes K_{X\times X}\big)
\cap\overset{2}{\bigwedge}H^0(K_X)
\ar@{^(->}[r]\ar@{^(->}[d]
&
H^0\big(\cal I(\Xi)\cdot\cal I\otimes K_{X\times X}\big)
\cap\overset{2}{\bigwedge} H^0(K_X)
\ar[r]^-{w_{X,\Delta_X}}\ar@{^(->}[d]
&
H^0(K_X^3(-\Delta_X))\ar@{=}[d]
\\ 
H^0\big(\cal I(DX)^2\otimes K_{X\times X}\big)\cap\overset{2}{\bigwedge} H^0(K_X)
\ar[r]\ar[d]^-{\ev_\Xi''}
&
H^0\big(\cal I\otimes K_{X\times X}\big)\cap\overset{2}{\bigwedge} H^0(K_X)
\ar@{->>}[r]^-{\wtld w\,'_X}\ar[d]^-{\ev_\Xi'}
&
H^0(K_X^3(-\Delta_X))
\\ 
\ouset{a=1}{\delta}{\bigoplus}K_{X,x_{a,1}}\otimes K_{X,x_{a,2}}\ar@{=}[r]
&
\ouset{a=1}{\delta}{\bigoplus}K_{X,x_{a,1}}\otimes K_{X,x_{a,2}}.
&
}
$}\hspace{2ex}
\end{equation}
A straightforward diagram chasing shows that $w_{X,\Delta_X}$ is surjective if 
\\[-.5ex]\centerline{$
\ev_\Xi'': \underbrace{H^0\bigl(\,\cal I(DX)^2\cdot K_{X\times X}\,\bigr)
\,\cap\,
\mbox{$\overset{2}{\bigwedge}$}\,H^0(K_X)}_{:=G}
\to 
\underbrace{
\mbox{$\ouset{a=1}{\delta}{\bigoplus}$}
K_{X\times X,(x_{a,1},\,x_{a,2})}}_{:=H_{\Xi}}
$}\\[-1ex] 
is so, or equivalently when the induced  
$h_\Xi:\overset{\delta}{\bigwedge}\,G\to\overset{\delta}{\bigwedge}\,H_\Xi$ 
is non-zero. This is indeed the case for generic markings. 

\nit{\it Claim.} $\;\underset{\Delta_X}{\bigcap}\Ker(h_\Xi)=0$. 
($h_\Xi$ depends on $\Delta_X$.) Indeed, since $\dim G\ges\delta$, we have 
\\ \centerline{
\scalebox{.8}{$
\xymatrix@C=1.5em@R=1.5em{
H^0(\cal I(DX)^2\otimes K_{X\times X})
\cap\overset{2}{\bigwedge}\,H^0(K_X)\;
\ar@{^(->}[r]\ar[d]^-{\ev_\Xi''}
&
H^0(\cal I\otimes K_{X\times X})
\cap\overset{2}{\bigwedge}\,H^0(K_X)\;
\ar@{^(->}[r]\ar[d]^-{\ev_\Xi'}
&
\overset{2}{\bigwedge}\,H^0(K_X)
\ar[d]^-{\ev_\Xi}
\\ 
\ouset{a=1}{\delta}{\bigoplus}K_{X\times X,(x_{a,1},\,x_{a,2})}
\ar@{=}[r]
&
\ouset{a=1}{\delta}{\bigoplus}K_{X\times X,(x_{a,1},\,x_{a,2})}
\ar@{=}[r]
&
\ouset{a=1}{\delta}{\bigoplus}K_{X\times X,(x_{a,1},\,x_{a,2})},
}
$}\\[.5ex]
}

\centerline{
$\kern1em
0\neq\overset{\delta}{\bigwedge}\,G
\subset
\overset{\delta}{\bigwedge}\bigl(\overset{2}{\bigwedge}H^0(K_X)\bigr)
\subset
H^0(K_{X\times X})^{\otimes \delta}
=
H^0\big((X^2)^{\delta},K_{X\times X}\boxtimes\ldots\boxtimes K_{X\times X}\big).
$\\[.5ex] 
}
\nit The wedge is a direct summand of the tensor product 
(appropriate skew-symmetric sums), and $h_\Xi$ is induced by the evaluation map 
\\[.5ex] \centerline{
$
\ev^{\delta}:
H^0\big((X^2)^\delta,K_{X\times X}\boxtimes\ldots\boxtimes K_{X\times X}\big)
\otimes\cal O
\to
K_{X\times X}\boxtimes\ldots\boxtimes K_{X\times X}
$
}
at $\big((x_{1,1},x_{1,2}),\ldots,(x_{\delta,1},x_{\delta,2})\big)\in (X^2)^\delta$. 
If $e\in\overset{\delta}{\bigwedge}G$ belongs to the intersection above, 
then $e\in H^0\big(\Ker(\ev^{\delta})\big)=\{0\}$. Hence, 
for any $e_1,\ldots,e_\delta\in G$ with $e_1\wedge\ldots\wedge e_\delta\neq 0$, 
there are markings $\Delta_X$ such that $\ev_\Xi''(e_1),\ldots,\ev_\Xi''(e_\delta)$ 
are linearly independent in $H_\Xi$ (thus they span it). 
\end{proof}


\section{Multiple point Seshadri constants of \mbox{$K3$} surfaces 
with cyclic Picard group}\label{sct:seshadri}

This section is independent of the rest. Here we determine a lower bound 
for the multiple point Seshadri constants of $\cA$, which is necessary for 
proving Lemma \ref{lm:surj}.

\begin{definition}(See \cite[Section 6]{de} for the original definition.) 
The multiple point Seshadri constant of $\cA$ corresponding to 
$\hat x_1,\ldots,\hat x_\delta\in S$ is defined as 
\begin{equation}\label{eq:seshadri}
\veps=\veps_{S,\delta}(\cA):=\inf_{Z}
\frac{Z\cdot\cA}{\mbox{$\ouset{a=1}{\delta}{\sum}\mlt_{\hat x_a}(Z)$}}
=
\sup\bigl\{
c\in\mbb R\mid \si^*\cA-cE\text{ is ample on }\tS
\bigr\}.
\end{equation}
The infimum is taken over all integral curves $Z\subset S$ which contain 
at least one of the points $\hat x_a$ above. Throughout this section we 
assume that $Z\in|z\cA|$, with $z\ges 1$.
\end{definition}

As the self-intersection number of any ample line bundle is positive, 
the upper bound $\veps\les\frac{\sqrt{\cA^2}}{\sqrt{\delta}}$ is automatic. 
We are interested in finding a lower bound. 

\begin{theorem}\label{thm:seshadri}
Assume that $\Pic(S)=\mbb Z\cA$, $\cA^2=2(n-1)\ges4$, and $\delta\ges1$. 
Then the Seshadri constant \eqref{eq:seshadri} satisfies 
$\veps\ges\frac{2\cA^2}{\delta+\sqrt{\delta^2+4\delta(2+\cA^2)}},$ 
for any points $\hat x_1,\ldots,\hat x_\delta\in S$.
\end{theorem}

Our proof is inspired from \cite{kn}, which treats the case $\delta=1$. 

\begin{proof}
We may assume that the points are numbered such that\\[.5ex] 
\centerline{$
\mlt_{\hat x_a}(Z)\ges 2,\text{ for }a=1,\ldots,\alpha,
\quad
\mlt_{\hat x_a}(Z)=1,\text{ for }a=\alpha+1,\ldots,\beta,
\quad (\beta\les\delta).
$}\\[.5ex]
We denote $p:=\ouset{a=1}{\alpha}{\sum}\mlt_{\hat x_a}(Z)\ges2\alpha$ 
and $m:=\ouset{a=1}{\delta}{\sum}\mlt_{\hat x_a}(Z)\les p+\delta-\alpha$. 

If $\alpha=0$, then $\frac{z\cdot\cA^2}{m}\ges\frac{\cA^2}{\delta}$ satisfies 
the inequality, so we may assume $\alpha\ges 1$. 
A point of multiplicity $m$ lowers the arithmetic genus of $Z$ by at least 
$\bigl(\begin{array}{c}\mbox{\scriptsize$m$}\\[-.5ex] 
\mbox{\scriptsize$2$}\end{array}\bigr)$, hence 
\\[.5ex] \centerline{$
p_a(Z)=\frac{z^2\cA^2}{2}+1\ges 
\frac{1}{2}\ouset{a=1}{\alpha}{\sum}
\bigl(\mlt_{\hat x_a}(Z)^2-\mlt_{\hat x_a}(Z)\bigr)
\ouset{\text{\scriptsize inequality}}{\text{\scriptsize Jensen}}{\ges}
\frac{1}{2}\bigl(\frac{p^2}{\alpha}-p\bigr), 
$}
so $p\les\frac{\alpha+\sqrt{\alpha^2+4\alpha(2+z^2\cA^2)}}{2}$. 
We deduce the following inequalities:\\[.5ex] 
\centerline{$
\frac{z\cA^2}{m}\ges\frac{z\cA^2}{p-\alpha+\delta}\ges
\underbrace{%
\frac{z\cA^2}{\delta+\frac{\sqrt{\alpha^2+4\alpha(2+z^2\cA^2)}-\alpha}{2}}%
}_{\text{\scriptsize decreasing in }\alpha}
\ges 
\underbrace{%
\frac{z\cA^2}{\delta+\frac{\sqrt{\delta^2+4\delta(2+z^2\cA^2)}-\delta}{2}}%
}_{\text{\scriptsize increasing in }z}
\ges 
\frac{2\cA^2}{\delta+\sqrt{\delta^2+4\delta(2+\cA^2)}}.
$}\\[-2ex]
\end{proof}

\begin{corollary}\label{cor:0} 
$H^0(\cal M)\to H^0(\cal M_E)$ is surjective, 
for $\cA^2\ges 6$ and $\delta\les\frac{d^2\cA^2}{3(d+4)}$. 
\end{corollary}

\begin{proof}
Indeed, it is enough to check that 
$H^1(\tS,\cal M(-E))= H^1(\tS,K_\tS\otimes\cal M(-2E))$ vanishes. By the 
Kodaira vanishing theorem, this happens as soon as 
$\cal M(-2E)=\si^*\cA^d(-3E)$ is ample. The previous theorem implies that, 
in order to achieve this, is enough to impose 
$\frac{3}{d}\les\frac{2\cA^2}{\delta+\sqrt{\delta^2+4\delta(2+\cA^2)}}$, 
which yields $\delta\les\frac{d^2(\cA^2)^2}{3(d\cA^2+3\cA^2+6)}$. 
\end{proof}


\section{Concluding remarks}\label{sct:final}

\nit\textbf{(I) Evidence for the non-surjectivity of $\bsymb{w_C}$}\quad 
Theorem \ref{thm:main} is a non-surjectivity property for the Wahl map of the 
\textit{pointed curve} $(C,\Delta)$, rather than of the curve $C$ itself. 

\nit\textit{Claim.} In order to prove the non-surjectivity of the Wahl map 
$w_C$, is enough to have the surjectivity of the evaluation homomorphism 
\begin{equation}\label{eq:evxi}
\mbox{$
H^0(\cal I(DC)^2\otimes K_{C\times C})
\to
\ouset{a=1}{\delta}{\bigoplus} K_{C\times C,(x_{a,1},x_{a,2})}
\oplus K_{C\times C,(x_{a,2},x_{a,1})}.
$}
\end{equation}
(For $\delta$ in the range \eqref{eq:dta}, corollary \ref{cor:0} implies that 
$K_C=\cal M_C$ separates $\Delta$, consequently 
$\overset{2}{\bigwedge}H^0(K_C)\to
\ouset{a=1}{\delta}{\bigoplus} K_{C\times C,(x_{a,1},x_{a,2})}$ is surjective. 
The surjectivity of \eqref{eq:evxi} yields that of $\ev_\Xi''$ in \eqref{eq:X}, 
which is relevant for us.) 

\nit For the claim, observe that one has the following implications 
(see \eqref{eq:DD}, \eqref{eq:X}): 
\begin{equation}\label{eq:implic}
w_C\text{ surjective }\;\Rightarrow\;
w_C'\text{ surjective }
\;\ouset{\rm surj.}{\eqref{eq:evxi}}{\Rightarrow}\;
w_{C,\Delta}\text{ surjective, a contradiction.}
\end{equation}

The surjectivity of \eqref{eq:evxi} is clearly a positivity property for 
$\cal I(DC)^2\otimes K_{C\times C}$. 
We use again the Seshadri constants to argue why this is likely to hold. 
The $\Xi$-pointed Seshadri constants of the self-product of a 
\emph{very general curve} $X$ at \emph{very general points} $\Xi$  
(as in Lemma \ref{lm:LP}) satisfy 
(see \cite[pp.65 below Theorem 1.6, and Lemma 2.6]{ross}): 
\begin{eqnarray}
\scalebox{.9}{$
\veps_{X\times X\!,\,\Xi}(\cal I(DX)^2\otimes K_{X\times X})
\ges 2(g-2)\veps_{\mbb P^2\!,\,g+\delta}(\cal O_{\mbb P^2}(1))
>\frac{2(g-2)}{\sqrt{g+\delta}}\sqrt{1-\frac{1}{8(g+\delta)}},
\phantom{=:\vphi(g,\delta)}
$}
\label{eq:s1}
\\ 
\scalebox{.9}{$
\veps_{X\times X\!,\,\Xi}(\cal I(DX)^4\otimes K_{X\times X})
\ges 4\cdot\frac{g-3}{2}\cdot\veps_{\mbb P^2\!,\,g+\delta}(\cal O_{\mbb P^2}(1))
>\frac{2(g-3)}{\sqrt{g+\delta}}\sqrt{1-\frac{1}{8(g+\delta)}}=:\vphi(g,\delta).
$}
\label{eq:s2}
\end{eqnarray}
The equation \eqref{eq:s2} implies (see \cite[Proposition 6.8]{de}) 
that ${\big(\cal I(DX)^2\otimes K_{X\times X}\big)}^2$ generates the jets 
of order $\lfloor\vphi(g,\delta)\rfloor-2$ at $\Xi\subset X\times X$. 
(We only need the generation of jets of order zero 
for $\cal I(DX)^2\otimes K_{X\times X}$; 
also, note that $\vphi(g,\delta)$ grows linearly with $\sqrt{g}$ as long as $\delta$ 
is small compared with $g$ (see \eqref{eq:dta}).) 
This discussion suggests that $\cal I(DX)^2\otimes K_{X\times X}$ is 
`strongly positive/generated'. However, the passage to \eqref{eq:evxi} above requires 
even more control. 
\medskip


\nit\textbf{(II) Related work}\quad 
The first version of this note was posted at http://www.arxiv.org over three weeks 
before the first posting of \cite{kem}. In this latter article, the author extensively 
studies the properties of nodal curves on $K3$ surfaces. 
Among several other results, he proves the non-surjectivity of a marked Wahl map 
(different from the one introduced in here) for nodal curves on $K3$ surfaces. 
\footnote{%
The proof of \cite[Theorem 1.7, pp. 32]{kem} is \emph{incomplete}: it claims the 
non-surjectivity of the modified Wahl map $w_{C,T}$ for {\emph{any}} nodal 
curve $C$ in a \emph{whole} (suitable) irreducible component of the moduli 
space $\cal V_{g,k}^n$ of nodal curves on polarized $K3$ surfaces $X$. The 
bottom line of the proof uses Lemma 3.17, that $H^0(C,f^*T_X)=0$. However,  
the lemma is proved (pp. 26) only for {\emph{general}} nodal curves $f:C\to X$. 
To conclude, one must use \textit{e.g.} the argument in the proof of 
the Theorem 1, General case, pp. \pageref{gen-case}, in this article.}


\end{document}